\documentclass[10pt, reqno]{amsart}
\usepackage[dvipsnames]{xcolor}
\usepackage{amsmath}
\usepackage{amsfonts}
\usepackage{amssymb}
\usepackage{subcaption}
\usepackage{amsthm}
\usepackage{tikz}
\usepackage[sc]{mathpazo}
\usepackage{bm}
\usepackage{graphicx}

\usepackage{algorithm}
\usepackage{algpseudocode}

\usepackage[font=small,labelfont=bf]{caption}

\usepackage{enumitem}
\setenumerate{itemsep=5pt} 
\setitemize{itemsep=5pt} 

\theoremstyle{plain}
\newtheorem{theorem}{Theorem}
\numberwithin{theorem}{section}

\newtheorem{corollary}[theorem]{Corollary}

\theoremstyle{definition}

\newtheorem{remark}[theorem]{Remark}
\newtheorem{example}[theorem]{Example}

\numberwithin{equation}{section}

\newcommand{\Ap}{\operatorname{Ap}}

\newcommand{\Z}{\mathbb{Z}}

\newcommand{\M}{\mathrm{M}}

\newcommand{\N}{\mathbb{N}}
\newcommand{\Q}{\mathbb{Q}}

\newcommand{\mdim}{\operatorname{dim}_{\mathrm{mat}}}
\renewcommand{\pmod}[1]{\,(\operatorname{mod} #1)}

\newcommand{\0}{{\color{lightgray}0}}

\renewcommand{\S}{\mathcal{S}}

\newcommand{\semigroup}[1]{\langle #1 \rangle}

\allowdisplaybreaks
\usepackage{hyperref}

\begin{document}

\title[Numerical semigroups from rational matrices II]{Numerical semigroups from rational matrices II: matricial dimension does not exceed multiplicity}

\author[A.~Chhabra]{Arsh Chhabra}
\address{Department of Mathematics and Statistics, Pomona College, 610 N. College Ave., Claremont, CA 91711, USA}
\email{acaa2021@mymail.pomona.edu}

\author[S.R.~Garcia]{Stephan Ramon Garcia}
\address{Department of Mathematics and Statistics, Pomona College, 610 N. College Ave., Claremont, CA 91711, USA}
\email{stephan.garcia@pomona.edu}
\urladdr{\url{https://stephangarcia.sites.pomona.edu/}}

\author[C.~O'Neill]{Christopher O'Neill}
\address{Mathematics Department, San Diego State University, San Diego, CA 92182}
\email{cdoneill@sdsu.edu}
\urladdr{\url{https://cdoneill.sdsu.edu/}}

\thanks{SRG partially supported by NSF grant DMS-2054002.}

\begin{abstract}
We continue our study of exponent semigroups of rational matrices.  Our main result is that the matricial dimension of a numerical semigroup is at most its multiplicity (the least generator), greatly improving upon the previous upper bound (the conductor).  For many numerical semigroups, including all symmetric numerical semigroups, our upper bound is tight.
\end{abstract}

\keywords{numerical semigroup; rational matrix; conductor}
\subjclass[2000]{20M14,05E40}

\maketitle

\section{Introduction}

Let $\M_d(\cdot)$ denote the set of $d \times d$ matrices with entries in the set $\Z$ of integers or the set $\Q$ of rational numbers, as indicated.
Every (additive) subsemigroup of $\N = \{0,1,2,\ldots\}$ is the \emph{exponent semigroup}
\begin{equation*}
\S(A) = \{ n \in \N : A^n \in \M_d(\Z)\}
\end{equation*}
of some $A \in \M_d(\Q)$ by \cite[Cor.~6.5]{NSRM1}.
In particular, every \emph{numerical semigroup} $S$, that is, a subsemigroup of $\N$ with finite complement \cite{Assi, Rosales}, 
is of the form $S = \S(A)$ with $A \in \M_c(\Q)$, in which $c=c(S) = 1 + \max(\N \backslash S)$ is the \emph{conductor} of $S$ \cite[Thm.~6.2]{NSRM1}.
This ensures that the \emph{matricial dimension}  
\begin{equation*}
\mdim S = \min\{ d \geq 1 : \text{there is an $A \in \M_d(\Q)$ such that $S = \S(A)$}\}
\end{equation*}
of a numerical semigroup $S \subseteq \N$ is well defined and $\mdim S \leq c(S)$.

Each numerical semigroup $S$ has a unique minimal system of \emph{generators}, that is, positive $n_1 <  n_2 < \cdots < n_k$ such that $S = \semigroup{n_1,n_2,\ldots,n_k}$ is the smallest additive subsemigroup of $\N$ containing $n_1, n_2,\ldots, n_k$.  Here, $e(S) = k$ is the \emph{embedding dimension} of $S$ and $m(S) = n_1$ is the \emph{multiplicity} of~$S$.

The main result of this paper is the following dramatic improvement of \cite[Thm.~6.2]{NSRM1}, in which $m(S) = n_1$ replaces $c(S)$.

\begin{theorem}\label{Theorem:Main}
If $S$ is a nontrivial numerical semigroup with multiplicity (minimal generator) $m(S)$, there is an $A \in \M_{m(S)}(\Q)$ such that $\S(A) = S$. 
Thus, $\mdim S \leq m(S)$.
\end{theorem}

While \cite[Thm.~6.2]{NSRM1} was algorithmic, with the bulk of the proof devoted to a proof of correctness, Theorem~\ref{Theorem:Main} gives an explicit construction in terms of Ap\'ery sets~\cite{apery}, a mainstay in the study of numerical semigroups~\cite{aperyunique,wilfconjecture}.  At the core of the proof lies a system of inequalities that reflect the fine structure of Ap\'ery sets; these were first introduced by Kunz in~\cite{kunz87}, and have since been utilized in enumerative~\cite{alhajjar,countingns,SystemInequalNumSemi} and classification~\cite{minprescard,kunzfaces1} problems in this area.

This paper is organized as follows.  Section~\ref{Section:Proof} contains the proof of Theorem~\ref{Theorem:Main}.
Section~\ref{Section:Examples} presents several illustrative examples and Corollary~\ref{Corollary:GeneralSemigroup}, which extends Theorem~\ref{Theorem:Main} to arbitrary semigroups in $\N$.
We close by computing the matricial dimension for irreducible numerical semigroups in Section~\ref{Section:Irreducible}.

\section{Proof of Theorem \ref{Theorem:Main}}\label{Section:Proof}
This section contains the proof of Theorem \ref{Theorem:Main} and some preliminary remarks.
Recall that the \emph{Ap\'ery set} of a numerical semigroup $S$ with multiplicity $m$ is the set
\begin{equation*}
\Ap(S) = \{n \in S : n - m \notin S\}.
\end{equation*}
Each element of $\Ap(S)$ is the smallest element of $S$ in its equivalence class modulo~$m$, so one often writes
\begin{equation*}
\Ap(S) = \{a_0,a_1, \ldots, a_{m-1}\},
\end{equation*}
in which $a_0 = 0$ and each $a_i \equiv i \bmod m$.  It is convenient to interpret the subscripts of the $a_i$ modulo $m$.
For example, it was shown by Kunz in~\cite{kunz87} that
\begin{equation}\label{eq:kunz}
a_i + a_j \geq a_{i+j}
\quad \text{for all} \quad i, j \in \Z
\end{equation}
and in fact this system of inequalities (along with the modular requirements and minimality of $m$) characterize Ap\'ery sets of numerical semigroups.  

Before proceeding to the proof of Theorem \ref{Theorem:Main}, let us first illustrate the matricial structure employed in the proof.  
Suppose $z_0,z_1,z_2,z_3 \neq 0$ and
\begin{equation*}
A = 
\begin{bmatrix}
 \0 & z_1 & \0 & \0 \\
 \0 & \0 & z_2 & \0 \\
 \0 & \0 & \0 & z_3 \\
 z_0 & \0 & \0 & \0 \\
\end{bmatrix}
\in \M_4(\Q),
\end{equation*}
which is a generalized permutation matrix or, equivalently, the adjacency matrix of a weighted directed cycle graph on $4$ vertices.
Observe that
\begin{align*}
A^{4q}
&=
\left[
\begin{smallmatrix}
 (z_0 z_1 z_2 z_3)^{q} & \0 & \0 & \0 \\
 \0 & (z_0 z_1 z_2 z_3)^{q} & \0 & \0 \\
 \0 & \0 & (z_0 z_1 z_2 z_3)^{q} & \0 \\
 \0 & \0 & \0 & (z_0 z_1 z_2 z_3)^{q} \\
\end{smallmatrix}\right],\\[5pt]
 A^{4q+1}
 &=
 \left[\begin{smallmatrix}
 \0 & z_1 (z_0 z_1 z_2 z_3)^q & \0 & \0 \\
 \0 & \0 & z_2 (z_0 z_1 z_2 z_3)^q & \0 \\
 \0 & \0 & \0 & z_3 (z_0 z_1 z_2 z_3)^q \\
 z_0 (z_0 z_1 z_2 z_3)^q & \0 & \0 & \0 \\
\end{smallmatrix}\right],\\[5pt]
A^{4q+2}
&=
\left[\begin{smallmatrix}
 \0 & \0 & z_1 z_2 (z_0 z_1 z_2 z_3)^q & \0 \\
 \0 & \0 & \0 & z_2 z_3 (z_0 z_1 z_2 z_3)^q \\
 z_3 z_0 (z_0 z_1 z_2 z_3)^q & \0 & \0 & \0 \\
 \0 & z_0 z_1 (z_0 z_1 z_2 z_3)^q & \0 & \0 \\
\end{smallmatrix}\right], \quad \text{and}\\[5pt]
 A^{4q+3}
 &=
 \left[\begin{smallmatrix}
 \0 & \0 & \0 & z_1 z_2 z_3 (z_0 z_1 z_2 z_3)^q \\
 z_3 z_0 z_2 (z_0 z_1 z_2 z_3)^q & \0 & \0 & \0 \\
 \0 & z_3 z_0 z_1 (z_0 z_1 z_2 z_3)^q & \0 & \0 \\
 \0 & \0 & z_0 z_1 z_2 (z_0 z_1 z_2 z_3)^q & \0 \\
\end{smallmatrix}\right].
\end{align*}
Writing $p = 4q+r$, with $q, r \in \Z$ and $0 \leq r \leq 3$,
the nonzero entries of $A^p$ are 
\begin{equation}\label{eq:entriesprod}
(z_0 z_1 z_2 z_3)^q \prod_{\ell = 0}^{r - 1} z_{i + \ell}
\quad \text{for}\quad 0 \leq i \leq 3,
\end{equation}
in which the subscripts are interpreted modulo $m$.  

For the proof of Theorem \ref{Theorem:Main},
we generalize \eqref{eq:entriesprod} to the $m \times m$ setting, and let each $z_i$ take the form $b^{x_i}$ with $b \in \Z \backslash\{-1,0,1\}$ and $x_i \in \Z$, so that the multiplicative structure of \eqref{eq:entriesprod} becomes additive.  

\begin{proof}[Proof~of Theorem.~\ref{Theorem:Main}]
Suppose $S$ is a numerical semigroup with multiplicity $m \geq 2$ and Ap\'ery set
$\Ap(S) = \{a_0, a_1, \ldots, a_{m-1}\}$,
in which $a_0 = 0$ and $a_i \equiv i \pmod m$ for each $i$.  
For each $i = 0, 1, \ldots, m-1$, let
\begin{equation*}
x_i = \tfrac{1}{m}(a_{i-1} - a_i + 1),
\end{equation*}
where the subscripts of the $x_i$ are, like those of the $a_i$, interpreted modulo $m$.  
Notice that each $x_i \in \Z$ since $a_{i-1} - a_i + 1 \equiv 0 \pmod m$, and telescoping yields
\begin{align}
x_0+ x_1 + \cdots + x_{m-1}
&= \tfrac{1}{m}(a_{m-1} - a_0 + 1) + \cdots  + \tfrac{1}{m}(a_{m-2} - a_{m-1} + 1) 
\nonumber \\
&= \tfrac{1}{m}(\underbrace{1 + \cdots + 1}_{m \text{ times}})
= 1.
\label{eq:SumOne}
\end{align}

Fix a base $b \in \Z \backslash \{-1,0,1\}$ and let $A \in \M_m(\Q)$ denote the matrix
\begin{equation}\label{eq:FinalMatrix}
A = 
\begin{bmatrix}
\0 & b^{x_1} & \0 & \0 & \color{lightgray}\cdots & \0  \\
\0 & \0 & b^{x_2} & \0 & \color{lightgray}\cdots & \0   \\
\0 & \0 & \0 & b^{x_3} & \color{lightgray}\cdots & \0   \\
\color{lightgray}\vdots & \color{lightgray}\vdots & \color{lightgray}\vdots & \color{lightgray}\vdots & \color{lightgray}\ddots & \color{lightgray}\vdots \\
\0 & \0 & \0 & \0 & \color{lightgray}\cdots & b^{x_{m-1}}  \\
b^{x_0} & \0 & \0 & \0 & \color{lightgray}\cdots & \0  \\
\end{bmatrix}.
\end{equation}
By~\eqref{eq:entriesprod}, for each $p \geq 0$, writing $p = qm + r$ with $q, r \in \Z$ and $0 \le r \le m-1$, the exponent of $b$ in each nonzero entry of $A^p$ has the form
\begin{equation*}
q + \sum_{\ell = 0}^{r-1} \tfrac{1}{m}(a_{i+\ell} - a_{i+\ell+1} + 1)
\quad \text{for some} \quad
i = 0, 1, \ldots, m-1.
\end{equation*}
As such, in order to prove $\S(A) = S$, we must show that the condition
\begin{equation}\label{eq:IntegralCondition}
q + \tfrac{1}{m}(a_i - a_{i+r} + r) \geq 0
\quad \text{for every} \quad
i = 0, 1, \ldots, m-1
\end{equation}
holds if and only if $p \in \S(A)$.  This clearly holds whenever $m \mid p$ since in this case $r = 0$ and $q \geq 0$.  By the definition of the Ap\'ery set, it suffices to prove that for each $j = 1, \ldots, m-1$, we have 
\begin{enumerate}[leftmargin=*]
\item 
$a_j \in \S(A)$ and 

\item 
$a_j - m \notin \S(A)$.  
\end{enumerate}
Indeed, if $p = a_j$, then for each $i = 0, 1, \ldots, m-1$, \eqref{eq:IntegralCondition} becomes
\begin{equation*}
\tfrac{1}{m}(a_j - j) + \tfrac{1}{m}(a_i - a_{i+j} + j)
= \tfrac{1}{m}(a_j + a_i - a_{i+j}) \geq 0,
\end{equation*}
wherein non-negativity follows from~\eqref{eq:kunz}.  Additionally, if $p = a_j - m$, then choosing $i = 0$ we get
\begin{equation*}
\tfrac{1}{m}(a_j - m - j) + \tfrac{1}{m}(a_0 - a_j + j)
= \tfrac{1}{m}(a_0 - m) = -1,
\end{equation*}
so condition~\eqref{eq:IntegralCondition} does not hold.  This completes the proof.  
\end{proof}

\section{Examples and remarks}\label{Section:Examples}

This section contains several remarks and illustrative examples of Theorem~\ref{Theorem:Main} that demonstrate the effectiveness of our main result, along with an extension of Theorem~\ref{Theorem:Main} to arbitrary subsemigroups of $\N$.  We begin with a careful analysis of how things play out for the so-called McNugget semigroup.  

\begin{example}\label{e:mcnugget}
Consider $S = \semigroup{6,9,20}$.  Then $m(S) = 6$ and $c(S) = 44$, so while the construction of \cite[Thm.~6.2]{NSRM1} produces a $B \in \M_{44}(\Q)$ such that $\S(B) = S$.  
In contrast, Theorem~\ref{Theorem:Main} produces an $A \in \M_6(\Q)$ such that $\S(A) = S$.  
Indeed, one can check that
$\Ap(S) = \{0, 49, 20, 9, 40, 29\}$,
so the proof of Theorem~\ref{Theorem:Main} yields
\begin{equation*}
\left[
\begin{array}{r}
x_0 \\
x_1 \\
x_2 \\
x_3 \\
x_4 \\
x_5 \\
\end{array}
\right]
=
\left[
\begin{array}{r}
5 \\
-8 \\
5 \\
2 \\
-5 \\
2 \\
\end{array}
\right]
\quad \text{and} \quad
A=
\begin{bmatrix}
\0 & \frac{1}{256} & \0 & \0 & \0 & \0 \\
\0 & \0 & 32 & \0 & \0 & \0 \\
\0 & \0 & \0 & 4 & \0 & \0 \\
\0 & \0 & \0 & \0 & \frac{1}{32} & \0 \\
\0 & \0 & \0 & \0 & \0 & 4 \\
32 & \0 & \0 & \0 & \0 & \0 \\
\end{bmatrix}.
\end{equation*}
In fact, this establishes $\mdim S = 6$ by Corollary~\ref{Corollary:Symmetric} below.  
\end{example}

\begin{example}
If $S = \semigroup{5, 11}$, then $\Ap(S) = \{0, 11, 22, 33, 44\}$, so Theorem~\ref{Theorem:Main} yields 
\begin{equation*}
A = 
\left[
\begin{smallmatrix}
\0 & b^{-2} & \0 & \0 & \0  \\
\0 & \0 & b^{-2} & \0 & \0   \\
\0 & \0 & \0 & b^{-2} & \0   \\
\0 & \0 & \0 & \0 & b^{-2}  \\
b^{9} & \0 & \0 & \0 & \0  \\
\end{smallmatrix}
\right].
\end{equation*}
More generally, if $S = \semigroup{m, km + 1}$ with $k \in \N$, then $\Ap(S)$ is comprised of integer mulitples of $km + 1$, 
so $a_i = i(km + 1)$ for $i = 1,2,\ldots,m-1$. The only nonzero integral entry of the resulting matrix is in the lower-left corner; the remaining nonzero entries are identical to each another.  
\end{example}

\begin{remark}\label{r:kunzcoords}
The values of $x_0,x_1, \ldots, x_{m-1}$ in the proof of Theorem~\ref{Theorem:Main} can also be expressed in terms of the so-called \emph{Kunz coordinates} $(k_1, k_2,\ldots, k_{m-1})$ of $S$, which are defined so that $a_i = k_im + i$ for each $i = 1,2, \ldots, m-1$ \cite{kunz87}.  In particular, 
\begin{equation*}
x_i = \begin{cases}
-k_1 & \text{if $i = 1$}, \\
k_{i-1} - k_i & \text{if $i = 2,3, \ldots, m-1$}, \\
k_{m-1} + 1 & \text{if $i = 0$}.
\end{cases}
\end{equation*}
Translating~\eqref{eq:kunz} in terms of Kunz coordinates requires the use of cases; this motivates the choice of expression for $x_i$ in the proof of Theorem~\ref{Theorem:Main}.  
\end{remark}

\begin{remark}\label{r:basearbitrary}
The parameter $b \in \Z \backslash\{-1,0,1\}$ in the proof of Theorem \ref{Theorem:Main} is arbitrary.
Laplace (cofactor) expansion of \eqref{eq:FinalMatrix} and equation \eqref{eq:SumOne} ensure that $\det A = (-1)^{m-1}b$,
so $\det A$ is arbitrary in $\Z \backslash\{-1,0,1\}$ and essentially independent of $S$.  
On the other hand, $\det A = \pm 1$ implies that $S$ is cyclic \cite[Thm.~4.2]{NSRM1}.  
Moreover, $\det A = 0$ whenever $A$ is nilpotent, and \cite[Thm.~6.2]{NSRM1} ensures that 
every numerical semigroup is the exponent semigroup of a nilpotent matrix.
\end{remark}

\begin{example}
A small adjustment to the proof of Theorem~\ref{Theorem:Main} permits one to find a representing matrix for any given subsemigroup of $\N$, numerical or not.  Let us consider $S = \semigroup{6,8,10}$.  Since $S = 2T$, in which $T = \semigroup{3,7,11}$ is a numerical semigroup, we have $T = \S(A)$ and $S = \S(B)$ for
\begin{equation*}
A =
\begin{bmatrix}
\0 & 2^{-2} & \0 \\
\0 & \0 & 2^{-1} \\
2^4 & \0 & \0 \\
\end{bmatrix}
\quad \text{and} \quad
B = 
\begin{bmatrix}
\0 & \frac{1}{3} & \0 & \0 & \0 & \0 \\
\0 & \0 & 3 \cdot 2^{-2} & \0 & \0 & \0 \\
\0 & \0 & \0 & \frac{1}{3} & \0 & \0 \\
\0 & \0 & \0 & \0 & 3 \cdot 2^{-1} & \0 \\
\0 & \0 & \0 & \0 & \0 & \frac{1}{3} \\
2^4 & \0 & \0 & \0 & \0 & \0 \\
\end{bmatrix}.
\end{equation*}
We record this observation here.  
\end{example}

\begin{corollary}\label{Corollary:GeneralSemigroup}
If $S \subseteq \N$ is an additive subsemigroup, then $\mdim S \le \min(S \backslash \{0\})$.  
\end{corollary}

\begin{proof}
Suppose $S = dT$ where $T$ is a numerical semigroup and $d \in \N$ is positive.  Let $m = m(T)$.  As in the proof of Theorem~\ref{Theorem:Main}, fix $x_0, x_1, \ldots, x_{m-1} \in \Z$ such that 
\begin{equation*}
A = 
\begin{bmatrix}
\0 & 2^{x_1} & \color{lightgray}\cdots & \0  \\
\color{lightgray}\vdots & \color{lightgray}\vdots & \color{lightgray}\ddots & \color{lightgray}\vdots \\
\0 & \0 & \color{lightgray}\cdots & 2^{x_{m-1}}  \\
2^{x_0} & \0 & \color{lightgray}\cdots & \0  \\
\end{bmatrix}
\end{equation*}
has $\S(A) = T$.  Let $B \in M_{dm}(\Q)$ be the matrix 
\begin{equation*}
B = 
\begin{bmatrix}
\0 & z_1 & \color{lightgray}\cdots & \0  \\
\color{lightgray}\vdots & \color{lightgray}\vdots & \color{lightgray}\ddots & \color{lightgray}\vdots \\
\0 & \0 & \color{lightgray}\cdots & z_{dm-1}  \\
z_0 & \0 & \color{lightgray}\cdots & \0  \\
\end{bmatrix}
\end{equation*}
defined by
\begin{equation*}
z_i = \begin{cases}
3^{d-1} 2^{x_k} & \text{if $i = kd$ with $k \in \Z$}, \\
\tfrac{1}{3} & \text{otherwise.}
\end{cases}
\end{equation*}
By \eqref{eq:entriesprod}, for any $p \in \N$ the nonzero entries in $B^{pd}$ are precisely those that appear in $A^p$, so $pd \in \S(B)$ if and only if $p \in \S(A)$.  On the other hand, any power of $B$ not divisible by $d$ has at least one non-integer entry with a power of 3 in the denominator, so $\gcd(\S(B)) = d$.  As such, we conclude $\S(B) = S$.  
\end{proof}

\section{Irreducible numerical semigroups}\label{Section:Irreducible}

Fix a numerical semigroup $S$, and let $F = c(S) - 1$.  
Recall that:
\begin{enumerate}[leftmargin=*]
\item 
$S$ is \emph{symmetric} if $x \in \Z \backslash S$ implies $F - x \in S$;

\item 
$S$ is \emph{pseudosymmetric} if $F$ is even and $x \in \Z \backslash S$ implies $F - x \in S$ or $x = F/2$; and

\item 
$S$ is \emph{irreducible} if $S$ is it cannot be written as an intersection of finitely many numerical semigroups properly containing it.  
\end{enumerate}

A numerical semigroup is irreducible if and only if it is symmetric or pseudosymmetric~~\cite{irreduciblecharacterization}, and these two families of numerical semigroups are each of interest in commutative algebraic settings (see~\cite{nsgorenstein} and~\cite{algebraicpseudosymmetric}, respecitvely).  

\begin{remark}\label{r:intersection}
Every numerical semigroup can be written as an intersection of finitely many irreducible numerical semigroups, and such expressions are often far from unique \cite{bogart2024unboundedness,nsirreducibledecomp}.  In some cases, one can use this fact and~\cite[Thm.~2.3(a)]{NSRM1} to obtain a more optimal construction than Theorem~\ref{Theorem:Main}.  
For example, Theorem~\ref{Theorem:Main} identifies $A \in \M_{15}(\Q)$ with exponent semigroup $\S(A) = \semigroup{15, 20, 21, 25, 26}$, but
\begin{equation*}
\S(A) = T \cap T'
\quad \text{with} \quad
T = \semigroup{3, 11}
\quad \text{and} \quad
T' = \semigroup{5, 16},
\end{equation*}
so one can obtain a block-diagonal matrix $A' \in \M_8(\Q)$ with $\S(A') = \S(A)$ by applying Theorem~\ref{Theorem:Main} to $T$ and $T'$.   Note that this strategy would be ineffective with the construction in~\cite[Thm.~6.2]{NSRM1} since $c(T \cap T') = \max(c(T), c(T'))$, while $m(T \cap T')$ can be much larger than $m(T) + m(T')$.  
\end{remark}

Remark~\ref{r:intersection} does not aid in obtaining the matricial dimension of irreducible numerical semigroups, since they cannot be written as an intersection of finitely many other numerical semigroups.  Luckily, Theorem~\ref{Theorem:Main} and~\cite{NSRM1} together identify the matricial dimension of nearly all such semigroups.  We record this here.  

\begin{corollary}\label{Corollary:Symmetric}
If $S$ is a symmetric numerical semigroup, then $\mdim S = m(S)$.  
In particular, if $e(S) = 2$, then $\mdim S = m(S)$.  
\end{corollary} 

\begin{proof}
This follows from Theorem \ref{Theorem:Main} since $\mdim S \geq m(S)$ by \cite[Thm.~5.3]{NSRM1}.  Additionally, $e(S) = 2$ implies $S$ is symmetric by \cite[Cor.~4.7]{Rosales}.
\end{proof}

\begin{corollary}\label{Corollary:Pseudosymmetric}
Let $S$ be a nontrivial pseudosymmetric numerical semigroup.  
\begin{enumerate}[leftmargin=*]
\item If $c(S) \le m(S)$, then $\mdim S = 2$.
\item If $m(S) < c(S) \le 2m(S)$, then $m(S) - 1 \leq \mdim S\leq m(S)$.
\item If $c(S) > 2m(S)$, then $\mdim S = m(S)$.
\end{enumerate}
\end{corollary}

\begin{proof}
Combine Theorem \ref{Theorem:Main} with the inequalities in \cite[Thm.~5.6]{NSRM1}.
\end{proof}

\begin{example}
The semigroup $S = \semigroup{3,5,7}$ is pseudosymmetric with $m(S) = 3$ and $c(S) = 5$.
One can readily check that $\S(A) = S$ for 
\begin{equation*}
A = \begin{bmatrix}
1 & -\tfrac{3}{16} \\[0.1em]
16 & \phantom{-}1
\end{bmatrix},
\end{equation*}
so $\mdim S = 2 = m(S) - 1$.  We conjecture that $\mdim S = m(S) - 1$ whenever $S$ is pseudosymmetric and $m(S) < c(S) \le 2m(S)$.  
\end{example}

\begin{example}
Consider $S = \semigroup{7,54,66}$, which has $m(S) = 7$, $c(S) = 192$, and 
\begin{equation*}
\Ap(S) = \{0, 120, 198, 66, 186, 54, 132\}.
\end{equation*}
Since $S$ is neither symmetric or pseudosymmetric, the previous corollaries do not determine $\mdim S$.  Theorem \ref{Theorem:Main} ensures that $\mdim S \leq 7$ whereas \cite[Thm.~6.2]{NSRM1} provides the much weaker bound $\mdim S \leq 192$.  However, we can prove that $\mdim S = 7$ as follows.
Suppose toward a contradiction that $S = \S(A)$, in which $A \in \M_d(\Q)$ with $1 \leq d \leq 6$.   
Then $185,186,187,188,189,190 \in S$ ensures that $\S(A)$ contains all successive natural numbers \cite[Thm.~5.1]{NSRM1}.
This contradicts the fact that $c(S) - 1 = 191 \notin S$.  Therefore, $\mdim S \geq 7$, so $\mdim S = 7$.
\end{example}

\begin{example}
These methods are insufficient to compute the matricial dimension of all numerical semigroups.  For example, 
the longest string of consecutive elements in $S = \semigroup{39, 40, 47}$ below $c(S) = 390$ is $\{351,352,\ldots,381\}$, which has length $31$.
Thus, $32 \leq \mdim S \leq 39$ by \cite[Thm.~5.1]{NSRM1} and Theorem~\ref{Theorem:Main}, respectively.  Since $S$ is neither symmetric nor pseudosymmetric, we cannot appeal to the corollaries above.
\end{example}

\bibliography{NSRM2}

\providecommand{\bysame}{\leavevmode\hbox to3em{\hrulefill}\thinspace}
\providecommand{\MR}{\relax\ifhmode\unskip\space\fi MR }
\providecommand{\MRhref}[2]{%
  \href{http://www.ams.org/mathscinet-getitem?mr=#1}{#2}
}
\providecommand{\href}[2]{#2}
\begin{thebibliography}{10}

\bibitem{alhajjar}
Elie Alhajjar, Travis Russell, and Michael Steward, \emph{Numerical semigroups
  and {K}unz polytopes}, Semigroup Forum, vol.~99, Springer, 2019,
  pp.~153--168.

\bibitem{apery}
Roger Ap\'ery, \emph{Sur les branches superlin\'eaires des courbes
  alg\'ebriques}, C. R. Acad. Sci. Paris \textbf{222} (1946), 1198--1200.
  \MR{17942}

\bibitem{Assi}
Abdallah Assi, Marco D'Anna, and Pedro~A. Garc\'{\i}a-S\'{a}nchez,
  \emph{Numerical semigroups and applications}, RSME Springer Series, vol.~3,
  Springer, Cham, [2020] \copyright 2020, Second edition [of 3558713].
  \MR{4230109}

\bibitem{algebraicpseudosymmetric}
Valentina Barucci, David~E. Dobbs, and Marco Fontana, \emph{Maximality
  properties in numerical semigroups and applications to one-dimensional
  analytically irreducible local domains}, Mem. Amer. Math. Soc. \textbf{125}
  (1997), no.~598, x+78. \MR{1357822}

\bibitem{bogart2024unboundedness}
Tristram Bogart and Seyed Amin~Seyed Fakhari, \emph{Unboundedness of
  irreducible decompositions of numerical semigroups}, arXiv preprint (2024),
  \url{https://arxiv.org/abs/2405.11080}.

\bibitem{NSRM1}
Arsh Chhabra, Stephan~Ramon Garcia, Fangqian Zhang, and Hechun Zhang,
  \emph{Numerical semigroups from rational matrices {I}: power-integral
  matrices and nilpotent representations},
  \url{https://arxiv.org/abs/2407.03560}.

\bibitem{minprescard}
Ceyhun Elmacioglu, Kieran Hilmer, Christopher O’Neill, Melin Okandan, and
  Hannah Park-Kaufmann, \emph{On the cardinality of minimal presentations of
  numerical semigroups}, Algebraic Combinatorics \textbf{7} (2024), no.~3,
  753--771.

\bibitem{countingns}
Nathan Kaplan, \emph{Counting numerical semigroups}, The American Mathematical
  Monthly \textbf{124} (2017), no.~9, 862--875.

\bibitem{kunzfaces1}
Nathan Kaplan and Christopher O'Neill, \emph{Numerical semigroups, polyhedra,
  and posets {I}: the group cone}, Combinatorial Theory \textbf{1} (2021).

\bibitem{nsgorenstein}
Ernst Kunz, \emph{The value-semigroup of a one-dimensional {G}orenstein ring},
  Proc. Amer. Math. Soc. \textbf{25} (1970), 748--751. \MR{265353}

\bibitem{kunz87}
\bysame, \emph{{\"U}ber die {K}lassifikation {N}umerischer {H}albgruppen},
  vol.~11, Fakult{\"a}t Mathematik der Universit{\"a}t Regensburg, 1987.

\bibitem{aperyunique}
J.~C. Rosales, \emph{Numerical semigroups with {A}p\'ery sets of unique
  expression}, J. Algebra \textbf{226} (2000), no.~1, 479--487. \MR{1749900}

\bibitem{nsirreducibledecomp}
J.~C. Rosales and M.~B. Branco, \emph{Decomposition of a numerical semigroup as
  an intersection of irreducible numerical semigroups}, Bull. Belg. Math. Soc.
  Simon Stevin \textbf{9} (2002), no.~3, 373--381. \MR{2016576}

\bibitem{irreduciblecharacterization}
\bysame, \emph{Irreducible numerical semigroups}, Pacific J. Math. \textbf{209}
  (2003), no.~1, 131--143. \MR{1973937}

\bibitem{Rosales}
J.~C. Rosales and P.~A. Garc\'{\i}a-S\'{a}nchez, \emph{Numerical semigroups},
  Developments in Mathematics, vol.~20, Springer, New York, 2009. \MR{2549780}

\bibitem{SystemInequalNumSemi}
J.~C. Rosales, P.~A. Garc\'{\i}a-S\'{a}nchez, J.~I. Garc\'{\i}a-Garc\'{\i}a,
  and M.~B. Branco, \emph{Systems of inequalities and numerical semigroups}, J.
  London Math. Soc. (2) \textbf{65} (2002), no.~3, 611--623. \MR{1895736}

\bibitem{wilfconjecture}
Herbert~S. Wilf, \emph{A circle-of-lights algorithm for the ``money-changing
  problem''}, Amer. Math. Monthly \textbf{85} (1978), no.~7, 562--565.
  \MR{556658}

\end{thebibliography}
\bibliographystyle{amsplain}

\end{document}